\documentclass[11pt]{amsart} 
\usepackage{fullpage, amsmath, amssymb, young}
\usepackage[all]{xypic}

\title{A new linear quotient of $\bC^4$ admitting a symplectic resolution} 
\usepackage{url} \usepackage{graphicx}

\numberwithin{equation}{subsection}

\theoremstyle{definition}
\newtheorem{theorem}[equation]{Theorem}
\newtheorem{lemma}[equation]{Lemma}
\newtheorem{proposition}[equation]{Proposition}
\newtheorem{corollary}[equation]{Corollary}

\newtheorem{remark}[equation]{Remark}

\newcommand{\Span}{\operatorname{Span}}
\newcommand{\Id}{\operatorname{Id}}
\newcommand{\rk}{\operatorname{rk}}
\newcommand{\Ad}{\operatorname{Ad}}

\newcommand{\Hilb}{\operatorname{Hilb}}

\newcommand{\HP}{\mathsf{HP}}
\newcommand{\bR}{\mathbf{R}}

\newcommand{\bZ}{\mathbf{Z}}

\newcommand{\iso}{{\;\stackrel{_\sim}{\to}\;}}

\newcommand{\cS}{\mathcal{S}}

\newcommand{\Aut}{\operatorname{Aut}}
\newcommand{\Out}{\operatorname{Out}}
\newcommand{\End}{\operatorname{End}}
\newcommand{\tr}{\operatorname{tr}}

\newcommand{\Ind}{\operatorname{Ind}}
\newcommand{\Rep}{\operatorname{\mathsf{Rep}}}

\newcommand{\GL}{\mathsf{GL}}
\newcommand{\SL}{\mathsf{SL}}
\newcommand{\Ort}{\mathsf{O}}
\newcommand{\Sp}{\mathsf{Sp}}

\newcommand{\gr}{\operatorname{\mathsf{gr}}}
\newcommand{\Spec}{\operatorname{\mathsf{Spec}}}

\newcommand{\onto}{\twoheadrightarrow}

\newcommand{\Sym}{\operatorname{\mathsf{Sym}}}

\newcommand{\mf}{\mathfrak}
\newcommand{\mc}{\mathcal}

\newcommand{\bC}{\mathbf{C}}

\newcommand{\msC}{\mathsf{C}}
\begin{document}
\date{2011}
\author{Gwyn Bellamy}
\address{School of Mathematics, Room 2.233, Alan Turing Building, University of Manchester, Oxford Road, Manchester, M13 9PL}
\email{gwyn.bellamy@manchester.ac.uk}

\author{Travis Schedler}
\address{MIT Department of Mathematics, Room 2-172,
77 Massachusetts Avenue
Cambridge, MA 02139-4307}
\email{trasched@gmail.com}



\begin{abstract}
  We show that the quotient $\bC^4 / G$ admits a symplectic resolution
  for $G = Q_8 \times_{\bZ/2} D_8 < \Sp_4(\bC)$.  Here $Q_8$ is the
  quaternionic group of order eight and $D_8$ is the dihedral group of
  order eight, and $G$ is the quotient of their direct product which
  identifies the nontrivial central elements $-\Id$ of each. It is
  equipped with the tensor product representation $\bC^2 \boxtimes
  \bC^2 \cong \bC^4$.  This group is also naturally a subgroup of the
  wreath product group $Q_8^2 \rtimes S_2 < \Sp_4(\bC)$. We compute the
  singular locus of the family of commutative spherical symplectic
  reflection algebras deforming $\bC^4/G$. We also discuss preliminary
  investigations on the more general question of classifying linear
  quotients $V / G$ admitting symplectic resolutions.
\end{abstract}
\maketitle
\tableofcontents

\section{Introduction and main results}
The quotients $V/G$, for $G < \Sp(V)$ a finite subgroup, which admit a
symplectic resolution (this notion is recalled in the next subsection)
are known to include:
\begin{enumerate}
\item[(i)] The type $A_n$ Weyl groups $S_{n+1}$, acting on $V =
  \bC^{2n} = T^* \bC^n$, where $\bC^n$ is the reflection
  representation; here a resolution is given by the Hilbert scheme
  $\Hilb^{n+1} \bC^2 / \bC^2$;
\item[(ii)] The wreath product groups $H^n \rtimes S_n$, for $H <
  \SL_2(\bC)$ a finite subgroup, acting on $\bC^{2n}$; here a
  resolution is given by the Hilbert scheme $\Hilb^n
  \widetilde{\bC^2/H}$, where $\widetilde{\bC^2/H} \to \bC^2/H$ is the
  minimal resolution of the Kleinian (or du Val) singularity
  $\bC^2/H$;
\item[(iii)] The exceptional complex reflection group $G_4 <
  \GL_2(\bC) < \Sp_4(\bC)$.
\end{enumerate}
The main purpose of this paper is to add one more example to this list:
\begin{enumerate}
\item[(iv)] The group $G = Q_8 \times_{\bZ/2} D_8$, where $Q_8 <
  \SL_2(\bC)$ is the quaternionic group of order eight, $D_8 <
  \Ort_2(\bC)$ is the dihedral group of order eight, and $Q_8
  \times_{\bZ/2} D_8$ is the quotient of their product which
  identifies the centers of $Q_8$ and $D_8$, acting on the tensor
  product representation $\bC^2 \boxtimes \bC^2$.
\end{enumerate}
As we will discuss briefly in \S \ref{ss:lsq} below, we suspect there
are few (if any) other examples remaining to be discovered.
\begin{remark}
  In cases (i) and (ii) above, one can construct the symplectic
  resolution in a natural way by a certain Hamiltonian reduction
  procedure.  On the other hand, in case (iii), we do not know of such
  a construction (although Lehn and Sorger constructed in \cite{LS}
  a resolution in a more explicit computational manner). We have also
  been unable to find such a construction for our new example (iv). To
  find such a construction seems like an interesting problem.
\end{remark}

In what follows, we will provide more detailed explanations of the
above and explain the proof that (iv) admits a symplectic resolution,
up to a computation given in \S \ref{s:proof}.

\subsection{Symplectic resolutions}
A symplectic resolution $\pi: \tilde X \to X$ of a (singular) variety
$X$ is a (smooth) symplectic variety $\tilde X$ equipped with a
proper, birational map $\pi$ to $X$. We are particularly interested in
the case that $X$ is affine; in this case $\pi$ can also be viewed as
an ``affinization'' of the symplectic variety $\tilde X$.  Such
structures have attracted a lot of interest in the last decade: see,
e.g., \cite{FuSurvey, Kal-gtsr}, and have strong applications to
representation theory, quantum algebra, algebraic geometry and
symplectic geometry.  Examples include the Springer resolution $T^*
(G/B) \to \mathcal{N}$ of the nilpotent cone $\mathcal{N}$ and its
Kostant-Slodowy slices, the Hilbert scheme $\Hilb^n(S)$ of $n$ points
on a symplectic surface $S$ resolving its $n$-th symmetric power
$\Sym^n(S)$, Nakajima quiver varieties, hypertoric varieties, and in
the case $S = \widetilde{\bC^2/G}$ is a minimal resolution of a
Kleinian (or du Val) singularity $\bC^2/G$, then $\Hilb^n(S)$ also
resolves the affine singularity $\Sym^n(\bC^2/G)$ (this is example
(ii) of the previous subsection).

The symplectic structure on $\tilde X$ naturally endows $X$ with a
Poisson structure.  Conversely, if $X$ is a Poisson variety, we say
that it admits a symplectic resolution if there exists a resolution
$\tilde X$ as above, such that $\pi$ is a Poisson morphism.  It is an
interesting question to determine which Poisson varieties admit symplectic
resolutions---this is a very strong condition.  On the other hand,
when such resolutions exist, they are derived unique: by
\cite{Kal-deq}, any two symplectic resolutions of a Poisson variety
have equivalent derived categories of coherent sheaves.

In the case $X = \bC^{2n}/G, G < \Sp_{2n}(\bC)$, the only known
examples where $X$ admits a symplectic resolution are the cases
(i)--(iii) of the previous subsection, and products thereof.

We exhibit a new example of a linear symplectic quotient
admitting a symplectic resolution: $\bC^4 / G$, where $G = Q_8
\times_{\bZ/2} D_8$, where $Q_8 < \SL_2(\bC)$ is the quaternionic
group of order eight, $D_8 < \Ort_2(\bR) < \Ort_2(\bC)$ is the
dihedral group of order eight, and $G$ is the quotient of their direct
product identifying the nontrivial central elements $-\Id$ of
each. This group $G$ is equipped with the faithful tensor product
representation $\bC^4 = \bC^2 \boxtimes \bC^2$.  Since $Q_8$ preserves
a symplectic form on $\bC^2$ and $D_8$ preserves an orthogonal form on
$\bC^2$, their product naturally preserves a symplectic form on the
tensor product $\bC^4$.  Thus $G$ is naturally a subgroup of
$\Sp_4(\bC)$.  This group can also be realized explicitly as the
following subgroup of the wreath product $Q_8^2 \rtimes S_2$:
\begin{equation}
  G = \{ (\pm g) \oplus g, ((\pm g) \oplus g) \sigma \mid g \in Q_8\} < 
  Q_8^2 \rtimes S_2,
\end{equation}
where $\sigma \in S_2$ is the nontrivial permutation. 

Our main result is then
\begin{theorem}\label{t:main1}
  The quotient $\bC^4 / G$ admits a symplectic resolution
  $\widetilde{\bC^4 / G} \to \bC^4/G$.
\end{theorem}
\subsection{Symplectic reflection algebras}
The proof of Theorem \ref{t:main1} is based on techniques from
\cite{EGsra} on the representation theory of symplectic reflection
algebras, together with a theorem of Namikawa \cite{Nam-pdasv}, and is
similar to that used in \cite[\S 7]{GordonBaby} and \cite[\S
4]{Belscms}. Namely, by Namikawa's result, since $\bC^4/G$ has a
contracting $\bC^\times$-action and the Poisson bracket has negative
degree with respect to this action, the existence of a symplectic
resolution follows from the existence of a smooth filtered Poisson
deformation of $\bC[V^*]^G$. Natural candidates for such a deformation
are the commutative spherical symplectic reflection algebras
$eH_c(G)e$ of \cite{EGsra}, where $H_c(G)$ is the symplectic
reflection algebra of \emph{op.~cit.} (with $t=0$, where $t$ is as in
\emph{op.~cit.} or \S \ref{ss:rsra} below), which deforms the skew
product algebra $\bC[V^*] \rtimes \bC[G]$, and $e \in \bC[G]$ is the
symmetrizer idempotent $e = \frac{1}{|G|}\sum_{g \in G} g$.
Conversely, by \cite[Corollary 1.21]{GiKal}, the existence of a
symplectic resolution of $V/G$ implies that the algebras $eH_c(G) e$
are generically smooth.  We deduce
\begin{theorem}\cite{GiKal,Nam-pdasv} The following conditions are
  equivalent:
\begin{enumerate}
\item[(i)] $V/G$ admits a symplectic resolution;
\item[(ii)] There exists a smooth commutative spherical symplectic
reflection algebra $eH_c(G)e$;
\item[(iii)] The algebras $eH_c(G)e$ are smooth for generic $c$.
\end{enumerate}
\end{theorem}
We remark that the equivalence of (ii) and (iii) is also clear since
smoothness is an open condition in $c$.  We will prove
\begin{theorem}\label{t:main2}
  For $G = Q_8 \times_{\bZ/2} D_8$, $eH_c(G)e$ is smooth for
  generic parameters $c$.
\end{theorem}
Later, in \S \ref{s:main3}, we will prove a much more general result,
which completely classifies the parameters $c$ for which the algebra
$e H_c(G) e$ is smooth (Theorem \ref{t:main3}), which turns out to be
the complement of exactly $21$ hyperplanes. There, we will also
describe in more detail the singular locus of the varieties $\Spec e
H_c(G) e$.

Recall that, for general $G < \Sp(V)$, commutative spherical symplectic
reflection algebras $eH_c(G)e$ are parameterized by
class functions $c: \bC[G] \to \bC$ (i.e., conjugation-invariant
functions) which are supported on the symplectic reflections $\cS
\subseteq G$, i.e., those elements $s \in G$ such that $s - \Id$ has
rank two.  In our example, there are five conjugacy classes of such
elements, so the parameter space is five-dimensional.

To prove Theorem \ref{t:main2}, we use the following reformulations of
smoothness for commutative spherical symplectic reflection algebras, 
at least some of which are probably well known:
\begin{proposition}\label{p:cssra-smooth}
  The following conditions are equivalent for a commutative spherical
symplectic reflection algebra $eH_c(G)e$:
\begin{enumerate}
\item[(i)] $eH_c(G)e$ is smooth;
\item[(ii)] $H_c(G)$ admits no irreducible representations which, as $G$-representations, are proper subrepresentations of the regular representation;
\item[(iii)] $H_c(G)$ admits no irreducible representations of
  dimension strictly less than $|G|$;
\item[(iv)] All finite-dimensional representations of $H_c(G)$ are, as
  $G$-representations, direct sums of finitely many copies of the
  regular representation;
\item[(v)] All irreducible representations of
  $H_c(G)$ restrict to the regular representation of $G$.
\end{enumerate}
\end{proposition}
\begin{proof}
  Clearly (v) $\Rightarrow$ (iv) $\Rightarrow$ (iii) $\Rightarrow$ (ii).
 By \cite[Theorem
  3.1]{EGsra}, there is a Satake isomorphism $Z(H_c(G)) \iso eH_c(G)
  e$ given by $z \mapsto z \cdot e$, where $Z(H_c(G))$ is the center
  of $H_c(G)$. Therefore (i) is equivalent to $Z(H_c(G))$ being
  smooth.

 By \cite[Theorem 1.7]{EGsra}, for every character $\eta: Z(H_c(G))
 \to \bC$ contained in the smooth locus of $\mathrm{Spec} (Z(H_c(G))$,
 the quotient $H_c(G) / \ker(\eta)$ is a matrix algebra with unique
 simple representation $H_c(G) e / \ker(\eta) e \cong
 \bC[G]$. Therefore (i) implies (v). 

 Since the smooth locus of $Z(H_c(G))$ is dense in $\Spec Z(H_c(G)$,
 \cite[Theorem 1.7]{EGsra} also implies that the P.I. degree of
 $H_c(G)$ equals $|G|$. Hence (iii) implies that the Azumaya locus of
 $H_c(G)$ equals $\Spec Z(H_c(G))$. However, it is known,
 e.g. \cite[Theorem 4.8]{GordonSurvey}, that the smooth locus of
 $Z(H_c(G))$ equals the Azumaya locus of $H_c(G)$ over
 $Z(H_c(G))$. Therefore (iii) implies (i).

 Using Lemma \ref{lem:regrep} below, we can show also that (ii)
 implies (i).\footnote{Since we will only actually need this
   implication for Theorem \ref{t:main3} and not for Theorem
   \ref{t:main2}, we postponed Lemma \ref{lem:regrep} used here to \S
   \ref{ss:regrep}.}  Suppose that (ii) holds.  By Lemma
 \ref{lem:regrep}, for every point of $\Spec Z(H_c(G)$, i.e.,
 every character $\eta$ of $Z(H_c(G))$, there exists a representation
 $M$ of $H_c(G)$ isomorphic to the regular representation with central
 character $\eta$.  By (ii), this must be irreducible. Because the
 P.I. degree of $H_c(G)$ equals $|G|$, again $\eta$ must be in the
 Azumaya locus and hence a smooth point. Thus (ii) implies (i).
\end{proof}
We will prove Theorem \ref{t:main2} by demonstrating that condition
(iii) holds for certain values of $c$ (and hence also for generic
$c$).  We will not need (ii) for Theorem \ref{t:main2}, but will use it in
the proof of the stronger Theorem \ref{t:main3}.

\subsection{Restrictions on the $G$-character of representations of
  symplectic reflection algebras}
To show that condition (iii) holds for generic $c$ (or equivalently,
some value of $c$), we exhibit sufficiently many restrictions on the
$G$-character $\chi$ of finite-dimensional representations of
$H_c(G)$. These restrictions apply to arbitrary symplectic reflection algebras.

For now, let $G < \Sp(V)$ be an arbitrary finite subgroup, for an
arbitrary symplectic vector space $V$.  Let $H_c(G)$ be a symplectic
reflection algebra deforming $\bC[V^*] \rtimes G$, and let $\rho: H_c(G)
\to \End(U)$ be a finite-dimensional representation. Whenever $x,y \in
V$, then $[x,y] \in \bC[G]$, and evidently $\tr(\rho([x,y]))= 0$. This
means that the character $\chi := \tr \circ \rho|_{\bC[G]}$ of $U$
annihilates all commutators $[x,y]$. These commutators are certain
explicit elements of $\bC[G]$ supported on $\cS$, that we will
describe later.

To show that $\chi$ must be a multiple of the regular character, i.e.,
that $\chi(g) = 0$ for all nontrivial $g$, such restrictions cannot be
sufficient unless all nontrivial elements of $G$ are symplectic
reflections. This only happens when $G < \SL_2(\bC)$.\footnote{On the
  other hand, in this case, one can indeed deduce that condition (iii)
  of Proposition \ref{p:cssra-smooth} holds for generic $c$, which are
  just class functions supported away from the trivial element of $G$,
  and this gives another proof of the well known fact that $\bC^2/G$
  admits a symplectic resolution.}  To obtain more restrictions, we
observe that, whenever $g \in G$, $x$ is a fixed vector of $g$, and $y
\in V$ is another element, then $[x, gy] =g[x,y] \in \bC[G]$, so that
$\chi(g[x,y]) = 0$ as well.  We have deduced:
\begin{proposition}\label{p:exp-ann}
  Let $H_c(G)$ be a symplectic reflection algebra associated to $G <
  \Sp(V)$. Let $g \in G$, $x \in V^g$, and $y \in V$. Then the element
  $g[x,y] \in H_c(G)$ lies in $\bC[G]$, and is annihilated by all
  characters of finite-dimensional representations of $H_c(G)$.
\end{proposition}
Then, the proof of Theorem \ref{t:main2}, and hence also Theorem
\ref{t:main1}, is completed by a straightforward computation of the
elements $g[x,y]$ that can arise in the case $G = Q_8 \times_{\bZ/2}
D_8$; together with the above proposition this will imply that
condition (iii) of Proposition \ref{p:cssra-smooth} holds.  We do this
in \S \ref{s:proof} below.
\begin{remark}
  The above proposition provides an algorithm for restricting the
  characters of finite-dimensional representations of $H_c(G)$ for
  generic $c$.  In fact, this was how we discovered our theorem in the
  first place. However, note that for the example of $G = G_4 <
  \GL_2(\bC) < \Sp_4(\bC)$, as computed in \cite[\S 4]{Belscms}, the
  algorithm above only restricts the $G$-representations to be a
  direct sum of copies of two representations (denoted $E$ and $F$ in
  \emph{op.~cit.}), of dimension less than $|G|$. Therefore these
  restrictions are not, in general, exhaustive, and do not give a
  necessary condition for $V/G$ to admit a symplectic resolution
  (since $\bC^4/G_4$ does admit a resolution by \cite{Belscms}).
\end{remark}

\subsection{On the (non)existence of symplectic resolutions for other
  linear symplectic quotients}\label{ss:lsq}
In this section, we explain what we know about the question of which
finite groups $G < \Sp(V)$ have the property that $V/G$ admits a
symplectic resolution, which we would like to address in future work.

By \cite{Verhsgos}, it is known that a linear symplectic quotient
$V/G$ by a finite subgroup $G < \Sp(V)$ can only admit a symplectic
resolution if $G$ is generated by symplectic reflections. In the case
that $G$ preserves a Lagrangian subspace $U$, so $G < \GL(U) <
\Sp(V)$, i.e., $G$ is a complex reflection group, these have a well
known classification by Shephard and Todd \cite{STfurg}. It was shown,
first for finite Coxeter groups in \cite{GordonBaby}, and then for
all complex reflection groups in \cite{Belscms} that, aside from one
exceptional group, denoted by $G_4$, only the infinite families
already mentioned (Weyl groups $S_{n+1}$ and wreath products
$(\bZ/m)^n \rtimes S_n$) have the property that $V/G$ admits a
symplectic resolution.

On the other hand, there are many groups generated by symplectic
reflections that are not complex reflection groups. These groups have
been classified in \cite{Cohsrc}.  Aside from finitely many
exceptional groups, they fall into infinite families. These infinite
families are subgroups of wreath products $\Gamma^n \rtimes S_n$,
where $\Gamma$ is some finite subgroup of $\SL_2(\bC)$ of type $D$ or
$E$: there are many types of such families for $\dim V = 4$, and a
few types of such families of increasing dimension $4, 6, 8, \ldots$.
Our group $G = Q_8 \times_{\bZ/2} D_8$ is included in the latter list
(it can be thought of as lying in an infinite family of either type).

Preliminary (but not definitive) computer evidence we have considered
seems to suggest that, for the infinite families involving $\dim V >
4$, and many of the infinite families in the case $\dim V = 4$, there
is no smooth commutative spherical symplectic reflection algebra
deforming $V/G$ and hence no symplectic resolution.  The problem
essentially reduces to the case of the families in $\dim V = 4$,
because the infinite families all contain parabolic subgroups $K < G$
such that $\dim (V^K)^\perp = 4$, and then one can adapt Losev's work
\cite{Los-csra} to show that, if $V/G$ admits a smooth deformation by a
commutative spherical symplectic reflection algebra, so must
$(V^K)^\perp/K$ as well.  In these cases, $K$ is in one of the
infinite families for the case of dimension four, so (except when $K$
is our group $Q_8 \times_{\bZ/2} D_8$), one reduces to showing that
$\bC^4/K$ admits no smooth deformation by a commutative spherical
symplectic reflection algebra.

We would guess that our group $G = Q_8 \times_{\bZ/2} D_8$ is the only
group in any of Cohen's aforementioned infinite families (aside from
the wreath products of groups in $\SL_2(\bC)$) such that $V/G$ admits
a symplectic resolution.  We do not presently have any understanding
of the (finitely many) exceptional symplectic reflection groups on
Cohen's list that are not complex reflection groups.

\subsection{Acknowledgements}
The second author is grateful to his coauthors of \cite{hp0bounds}
for discussions about the group $G$ considered here, which led to the
discovery of Proposition \ref{p:hp0-cssra} and partially motivated
this work.  The first author is supported by the EPSRC grant
EP-H028153. The second author is a five-year fellow of the American
Institute of Mathematics, and was partially supported by the
ARRA-funded NSF grant DMS-0900233.  The authors would like to thank
the Max-Planck Institute for Mathematics in Bonn, where some of this
work was done, for hospitality.

\section{The group $Q_8 \times_{\bZ/2} D_8$}
It is useful to describe the group $G = Q_8 \times_{\bZ/2} D_8$ in
some more detail---it turns out to enjoy some remarkable properties.

Let $i \in \bC$ denote the usual ``imaginary'' number, i.e., $i^2 = -1$.
Let
\[
Q_8 := \{\pm \Id, \pm I, \pm J, \pm K \mid IJ = K, JK = I, KI = J,
I^2=J^2=K^2=-\Id\} < \SL_2(\bC)
\]
be the usual description of $Q_8$. A faithful representation is given by: 
\[
I = \begin{pmatrix} i & 0 \\ 0 & -i \end{pmatrix}, \  J = \begin{pmatrix}
  0 & -1 \\ 1 & 0 \end{pmatrix}, \ K = \begin{pmatrix} 0 & -i \\ -i &
  0 \end{pmatrix}.
\]

Let 
\[
D_8 := \{\Id, \rho, \rho^2, \rho^3, \sigma, \sigma \rho, \sigma
\rho^2, \sigma \rho^3 \mid \sigma^2=\Id=\rho^4, \sigma \rho \sigma =
\rho^{-1}\} < \Ort_2(\bC)
\] 
be the usual description of $D_8$. A faithful representation is given by: 
\[
\rho = \begin{pmatrix} 0 & -1 \\ 1 & 0 \end{pmatrix}, \ \sigma
= \begin{pmatrix} 0 & 1 \\ 1 & 0 \end{pmatrix}.
\]
Note that the centers of $Q_8$ and $D_8$ are both $\pm \Id = \rho^2$
(which also coincide with the subgroup of scalar matrices, since
$\bC^2$ is an irreducible representation of both).  This makes $Q_8
\times_{\bZ/2} D_8$ act on $\bC^2 \otimes \bC^2 \cong \bC^4$,
preserving the product of the symplectic form on the first factor and
the orthogonal form on the second factor (as pointed out in the
introduction). That is, it preserves a symplectic form on $\bC^4$, and
this identifies $G := Q_8 \times_{\bZ/2} D_8 < \Sp_4(\bC)$.  We will
refer to the defining representation $\bC^4$ as the \emph{symplectic
  reflection representation}. It is clear that it is irreducible.

We now collect the facts we will need about $G$:
\begin{proposition}\label{p:Gfacts}
\begin{enumerate}
\item[(i)] All conjugacy classes of $G$, except for $\{\Id\}$ and
  $\{- \Id\}$, are of order two and of the form $\{\pm (g,h)\}$.
\item[(ii)] The symplectic reflections in $Q_8 \times_{\bZ/2} D_8$ are
  the noncentral elements $(g,h)$ where $g \in Q_8$ and $h \in D_8$
  have the same order (two or four).
\item[(iii)] Equivalently, the symplectic reflections are exactly the noncentral
  elements of order two.
\item[(iv)] Explicitly, there are five conjugacy classes of symplectic
  reflections:
\[
\{\pm (I, \rho)\}, \{\pm (J, \rho)\}, \{\pm (K,
  \rho)\}, \{\pm (\Id, \sigma)\}, \{\pm (\Id, \sigma \rho)\}.
\]
\item[(v)] The group $G$ has seventeen irreducible representations
  over $\bC$; sixteen of them are one dimensional and the other is the
  symplectic reflection representation $\bC^4$.
\end{enumerate}
\end{proposition}
\begin{proof}
  (i) It is clear that the conjugacy class $\Ad G\{(g,h)\}$ containing
  an element $(g,h)$ is the product of conjugacy classes of $(g,1)$
  and $(1,h)$, i.e., $\Ad(Q_8)\{g\} \times_{\bZ/2} \Ad(D_8)\{h\}$. The
  statement follows from the fact that it holds for each of $Q_8$ and
  $D_8$.

  (ii,iv) The eigenvalues of $(g,h)$ are the four pairwise products of
  an eigenvalue of $g$ and an eigenvalue of $h$. In order for the
  result to contain one as an eigenvalue, therefore, $g$ and $h^{-1}$
  must share a common eigenvalue.  In this case, this can only happen
  if the eigenvalues of $g$ and $h$ are both $i$ and $-i$ (i.e., $g$
  and $h$ both have order four), or if $g = \pm \Id$ and $h \in
  \{\sigma, \sigma \rho, \sigma \rho^2, \sigma \rho^3\}$.

  (iii) Note that, if one of $g$ and $h$ has order four, but the other
  has order two, then $(g,h)^2 = -\Id$, so $(g,h)$ has order four as
  well.  So the description follows.
\end{proof}
\subsection{Outer automorphisms of $G$}
The material of this section will not be needed in the paper, but we
are including it to demonstrate the unique symmetry of $G$ (which,
along with properties already described, makes it appear somewhat
exceptional).
\begin{proposition}\label{p:outgs5}
\begin{enumerate}
\item[(i)] The permutation action of $\Out(G)$ on the conjugacy
  classes of symplectic reflections defines an isomorphism
\begin{equation}\label{e:outgs5}
\Out(G) \iso S_5.
\end{equation}
\item[(ii)] All of the outer automorphisms are obtainable by
  conjugation by elements of $\Sp_4(\bC)$.
\item[(iii)] This outer automorphism group is generated by the outer
  automorphism group of $D_8$ along with the conjugation action of
  $Q_8^2 \rtimes S_2$.
\end{enumerate}
\end{proposition}
\begin{proof} 
 (i) In the realization $G = Q_8 \times_{\bZ} D_8$, one sees the
 subgroup of the outer automorphism subgroup $\Out(Q_8) \times
 \Out(D_8) < \Out(G)$ of order $12$.  On the other hand, in the
 realization $G < Q_8^2 \rtimes S_2$, one sees the subgroup of outer
 automorphisms coming from conjugation by the larger group. Since
 $\bC^4$ is an irreducible representation of $G$, the centralizer of
 $G$ in $Q_8^2 \rtimes S_2$ is only the scalar matrices $\pm \Id$, so
 in this way one obtains the subgroup of outer automorphisms of order
 $4$ (in particular, it is $\bZ/2 \times \bZ/2$).

  We claim that these two groups do not intersect nontrivially, and
  their permutation actions on conjugacy classes of symplectic
  reflections generate all of $S_5$. To see this, note first that, in
  the realization $G = Q_8 \times_{\bZ/2} D_8$, the outer automorphism
  subgroup $\Out(Q_8) \times \Out(D_8)$ preserves the partition of symplectic reflection conjugacy
  classes into the cells
\[
\{\{\pm(I,\rho)\},\{\pm(J,\rho)\}, \{\pm(K,\rho)\}\}, \text{ and }
  \{\{\pm(\Id,\sigma)\},\{\pm \Id, \sigma\rho\}\}.
\]
 In fact, this
  produces an isomorphism
\[
\Out(Q_8) \times \Out(D_8) \iso S_3 \times S_2 < S_5,
\] 
by permutations of symplectic reflection conjugacy classes.

On the other hand, in the realization $G < Q_8^2 \rtimes S_2$, the
outer automorphism subgroup coming from the conjugation action of
$Q_8^2 \rtimes S_2$ preserves the partition of symplectic reflection
conjugacy classes into the cells
$$
\{\{\pm(I,\rho)\},\{\pm(J,\rho)\}, \{\pm(K,\rho)\}, \{\pm(\Id,
  \sigma)\}\}, \textrm{ and } \{\{\pm(\Id, \sigma \rho)\}\}
$$
(note that the last conjugacy class is the one consisting of the
noncentral diagonal matrices). This produces an isomorphism
  \[
\Ad(Q_8^2 \rtimes S_2) / \Ad(G) \iso \bZ/2 \times \bZ/2 < S_4 < S_5,
\]
again by permuting the symplectic reflection conjugacy classes.

It is then clear that the above two groups generate all of $S_5$.

To prove the assertion, it remains to show that one obtains from this
an isomorphism $\Out(G) \to S_5$, by permutating the symplectic
reflection conjugacy classes.

First, we have to explain why all outer automorphisms preserve the
conjugacy classes of symplectic reflections.  This follows because the
symplectic reflections are exactly the noncentral involutions
(Proposition \ref{p:Gfacts}.(iii)). Alternatively, since the defining
representation $\bC^4$ of $G$ is the unique four-dimensional
irreducible representation, any outer automorphism must be obtained by
conjugation by an element of $\GL_4(\bC)$, so that the symplectic
reflections (elements $g$ such that $(\bC^4)^g$ is two-dimensional)
must be preserved.

Hence, the above yields a well defined epimorphism $\Out(G) \onto
S_5$.  It remains to show that this is injective, i.e., the kernel of
$\Aut(G) \to S_5$ is the inner automorphism group. It is clear that
the inner automorphism group is contained in the kernel, so we only
have to show it equals the kernel.  Any automorphism which fixes all
the symplectic reflection conjugacy classes is determined by how it
acts on each of the classes (since $G$ is generated by symplectic
reflections).  There can be at most $32$ of these, and it suffices to
show there are only $16 = |G/Z(G)|$ of them. However, any four of
these conjugacy classes generates the fifth, which implies that there
can be at most $16$. Hence there are exactly $16$ and the kernel of
$\Aut(G) \onto S_5$ is the inner automorphism group, as desired.

(iii) This follows from the proof of (i): we pointed out that all of
the mentioned elements generate the whole outer automorphism group
$S_5$. But more precisely, we did not actually need the outer
automorphism group of $Q_8$: the outer automorphism group of $D_8$
provides the transposition in $S_5$, and this together with the
order-four subgroup of $S_4 < S_5$ (where $S_4$ does not contain the
aforementioned transposition) generates all of $S_5$.

(ii) This follows from (iii) if we can just show that the nontrivial
element of $\Out(D_8) \cong \bZ/2$, as a subgroup of $\Out(G)$, is
obtainable by conjugation by an element of $\Sp_4(\bC)$ (note that
$Q_8^2 \rtimes S_2 < \Sp_4(\bC)$, which proves that the conjugation
action of the latter is by symplectic transformations).  This element
is the automorphism $\sigma \mapsto \sigma \rho, \rho \mapsto \rho$,
of order four as an honest automorphism (as an outer automorphism it has order two).  It suffices to show that this automorphism of
$D_8$ is given by conjugation by an element of $\Ort_2(\bC)$. This can
be done by conjugating by any square root of $\rho$, which is indeed
orthogonal.
\end{proof}
\begin{remark}
Alternatively, to show that the automorphism of $D_8$ is given by conjugation
by an element of $\Ort_2(\bC$),
one can argue that, since $\bC^2$ is the unique
irreducible representation of $D_8$ of dimension $2$, the outer
automorphism is given by conjugating by some matrix, and this can be
taken to be orthogonal since it can be taken to be real (there is only
one real irreducible representation of dimension two).

A similar argument applied to $G$ yields a proof of all of part (ii):
the irreducible representation $\bC^4$ is the unique one of dimension
four, so any outer automorphism is obtained by conjugation by some
element of $\GL_4(\bC)$. In fact, this is the unique irreducible
\emph{symplectic} representation of dimension four, since all the
other irreducible representations of $G$ are one-dimensional and
extend to two-dimensional irreducible symplectic representations (the
symplectic representation theory of any finite group is completely
reducible just like the ordinary representation theory). Thus, any
outer automorphism must be given by conjugating by an
element of $\Sp_4(\bC)$.
\end{remark}

\section{Proof of Theorem \ref{t:main2}}\label{s:proof}
\subsection{Recollections on symplectic reflection algebras
 (following \cite{EGsra})} \label{ss:rsra}
Let $(V,\omega)$ be a symplectic vector space. Recall that a
\emph{symplectic reflection} is an element $s \in \Sp(V)$ such that
$\rk(s - \Id) = 2$, i.e., $V^s \subseteq V$ is a codimension-two
subspace, which we call the reflecting hyperplane of $s$.  The restriction of $\omega$ to $V^s$ is nondegenerate, so $V = V^s \oplus (V^s)^\perp$.  Let
$\pi_{(V^s)^\perp}: V \onto (V^s)^\perp$ be the orthogonal (with
respect to $\omega$) projection.  Define the (degenerate on $V$) form
\[
\omega_s: V \otimes V \to \bC, \quad \omega_s(v,w) =
\omega(\pi_{(V^s)^\perp}(v), \pi_{(V^s)^\perp}(w)).
\]
Now, let $G < \Sp(V)$ be a finite subgroup.  Let $\cS \subseteq G$ be
the subset of symplectic reflections.  Let $\msC = \bC[\cS]^G$ denote the set
of conjugation-invariant functions on $\cS$.  For every $c \in \msC$
and $t \in \bC$, define the \emph{symplectic reflection algebra}
\[
H_{c,t}(G) := T V \rtimes G / (v \cdot w - w \cdot v - t\omega(v,w) -
\sum_{s \in S} c(s) \omega_s(v,w)),
\]
where $TV$ is the tensor algebra on $V$ (with multiplication
$\cdot$). As in the introduction, let $e \in \bC[G]$ be the
symmetrizer element $e := \frac{1}{|G|} \sum_{g \in G} g$, and define
the spherical symplectic reflection algebra as $e H_{c,t}(G) e$.

We will be interested in the case $t=0$, and will use the notation
$H_c(G) := H_{c,0}(G)$. In this case, it is a well known result of
\cite{EGsra} that $eH_c(G)e$ is commutative and is in fact isomorphic to
the center of $H_c(G)$. Therefore, we call $e H_c(G) e$ a commutative spherical symplectic reflection algebra.

\subsection{Proof of Theorem \ref{t:main2}} 

As before, set $G := Q_8 \times_{\bZ/2} D_8$.  We will prove in the
next subsection the following more precise result, using Proposition
\ref{p:exp-ann}:
\begin{proposition}\label{p:hyp-eqns}
  Let $\chi$ be the $G$-character of a finite-dimensional
  representation of $H_c(G)$. Then the following equations:
\begin{equation}\label{e:hyp-eqns1}
\sum_{s \in \cS} \chi(s) \cdot c(s) = 0,
\end{equation}
and, for all $g \in \cS$:
\begin{equation}\label{e:hyp-eqns2}
2 \chi(- \Id)c(g) + \sum_{s \in \cS \setminus \{ g,-g \}} \chi(gs) \cdot c(s) = 0.
\end{equation}
are satisfied by $\chi$.
\end{proposition}

In view of Proposition \ref{p:cssra-smooth}.(iii), we immediately
conclude
\begin{corollary}\label{c:hyp-eqns}
  If equations \eqref{e:hyp-eqns1} and \eqref{e:hyp-eqns2} are not
  satisfied for any character $\chi$ of a representation of dimension
  less than $|G|$, then $e H_c(G) e$ is smooth.
\end{corollary}
We may conclude from this Theorem \ref{t:main2}:
\begin{proof}[Proof of Theorem \ref{t:main2}]
  First note that, by Proposition \ref{p:Gfacts}, for all $h \in G$,
  either $h \in \cS$, or $h = gs$ for some $g, s \in \cS$ (and if $h \neq
  \Id$, then $g \neq s$; recall $s = s^{-1}$ for all $s \in \cS$).
  Therefore, at least one of equations
  \eqref{e:hyp-eqns1}-\eqref{e:hyp-eqns2} is non-trivial unless
  $\chi(h) = 0$ for all $h \neq \Id$, i.e., $\chi$ is a multiple of
  the regular character. As a result, \eqref{e:hyp-eqns1} and
  \eqref{e:hyp-eqns2} define proper linear subspaces of $\bC[\cS]^G$ as
  $\chi$ ranges over all characters of finite-dimensional
  representations of dimension less than $|G|$.  Hence, for $c$ not in
  any of these finitely many proper linear spaces (and in particular
  for generic $c$), Corollary \ref{c:hyp-eqns} implies that $e H_c(G)
  e$ is smooth.
\end{proof}
We remark that the equations \eqref{e:hyp-eqns1} and
\eqref{e:hyp-eqns2} for $c$ have integer coefficients since all
characters of $G$ are integer-valued (and characters of
representations of dimension $< |G|$ are valued in integers of
absolute value less than $|G|$).
\subsection{Proof of Proposition \ref{p:hyp-eqns}}

If $\chi$ is the character of a representation of $H_c(G)$, then
Proposition \ref{p:exp-ann} implies that $\chi(g[x,y]) = 0$ whenever
$x \in V^g$ and $y \in V$.

Choose $x,y \in V$ such that $\omega(x,y) = 2$. 
For $s \in \mathcal{S}$, since $s^2 = \Id$, we conclude that $(V^s)^\perp =
V^{-s}$, and hence $\omega_s + \omega_{-s} = \omega$. Since also the conjugacy class of $s$ is $\{s, -s\}$, we conclude that 
\[
\chi([x,y])=\sum_{s \in \cS} \omega_s(x,y) c(s) \chi(s) = 
\frac{1}{2} \omega(x,y) \sum_{s \in \cS} c(s) \chi(s) = 0,
\]
which
equals \eqref{e:hyp-eqns1}.

Next, fix $g \in \cS$. Then, $V^g \neq 0$. Let $x \in V^g$ and $y \in V$
be such that $\omega(x,y) = 2$.  Then,
\begin{equation}\label{e:gs}
\chi(g[x,y]) = \sum_{s \in \cS} \omega_s(x,y) c(s) \chi(gs).
\end{equation}
Let $S \subseteq \mathcal{S}$ be a conjugacy class of symplectic reflections.
If $g \notin S$ then Proposition \ref{p:Gfacts}.(i) implies that $g
\cdot S$ is again a conjugacy class in $G$. Hence
\[
\omega_s(x,y) c(s) \chi(gs) + \omega_{-s}(x,y) c(-s) \chi(-gs) =
\omega(x,y) c(s) \chi(gs) = c(s) \chi(gs) + c(-s)
\chi(-gs).
\]
If, on the other hand, $g \in S$, then $S = \{ g, -g \}$, and our
choice of $x$ implies that $\omega_g(x,y) = 0$. Therefore
$\omega_{-g}(x,y) = \omega(x,y)$ and hence
$$
\omega_{g}(x,y) c(g) \chi(g^2) + \omega_{-g}(x,y) c(-g) \chi(-g^2) =
2c(-g) \chi(- \Id).
$$
Put together, \eqref{e:gs} becomes
\[
\chi(g[x,y]) = 2c(-g) \chi(- \Id) + 
\sum_{s \in \cS \setminus \{g,-g\}} c(s) \chi(gs),
\]
implying \eqref{e:hyp-eqns2}.
 
\section{The singular locus of $e H_c(G) e$}\label{s:main3}
It turns out to be possible to completely characterize the locus of $c
\in \msC$ such that $e H_c(G) e$ is singular, generalizing Theorem
\ref{t:main2} (see Theorem \ref{t:main3} below).  Before we do this,
we recall some elementary facts about symplectic leaves, which are not
strictly needed for the theorem, but which we will use to describe in
more detail the singularities of those commutative spherical
symplectic reflection algebras that are singular.

\subsection{Recollections on symplectic leaves} \label{ss:symleaves}
Recall that an (algebraic) symplectic leaf of an affine Poisson
variety $X$ is a (Zariski) locally closed and connected smooth
subvariety $Y$ such that the tangent space $T_y Y$ at each point $y
\in Y$ is spanned by Hamiltonian vector fields, $\xi_f := \{f, -\}$,
for $f \in \bC[X]$.  The symplectic leaves are all symplectic manifolds
(with Poisson structure obtained from the Poisson structure on $X$),
and are in particular even-dimensional.  When a Poisson variety $X$ is
a union of finitely many (necessarily disjoint) symplectic leaves,
then this decomposition is unique. Moreover, the singular locus of $X$
is exactly the union of those leaves that are not open in $X$ (i.e.,
the positive-codimension leaves when $X$ is irreducible). (We remark
that this property of being a finite union of symplectic leaves is, in
general, a strong condition, which was studied in, e.g.,
\cite{Kalss,ESdm}; note that it is always satisfied for varieties admitting
a symplectic resolution.)

For every finite subgroup $G< \Sp(V)$, the Poisson variety $\Spec
\bC[V^*]^G = V^* / G$ is a union of finitely many symplectic leaves,
which are the $G$-orbits of the parabolic subspaces $V^K \subseteq V$
for subgroups $K < G$ (see, e.g., \cite[Proposition
7.4]{BrownGordon}).  Thus, for any filtered Poisson deformation $A$ of
$\bC[V^*]^G$, it is also true that $\Spec A$ has finitely many
symplectic leaves: for each $i \geq 0$, the union of the $\leq
2i$-dimensional leaves corresponds to a Poisson ideal $J \subseteq A$
whose associated graded Poisson ideal $\gr(J)$ can only vanish on
$\leq 2i$-dimensional leaves of $V^*/G$.  In particular, since
$\gr(J)$ is $\leq 2i$-dimensional, so is $J$, and hence there can only
be finitely many $2i$-dimensional symplectic leaves of $\Spec A$.

Therefore, in our situation where $V = \bC^4$, describing the
singularities of each commutative spherical symplectic reflection
algebra deforming $\bC[V^*]^G$ is equivalent to
determining all two-dimensional and all zero-dimensional symplectic
leaves.  

Below, for our group $G = Q_8 \times_{\bZ/2} D_8$, in addition to
describing completely the set of parameters $c \in \mathsf{C}$ for
which the corresponding algebra $e H_c(G) e$ is smooth (which by
Theorem \ref{t:main2} forms an open subvariety of the parameter
space), we will describe (and enumerate) all two-dimensional
symplectic leaves of all commutative spherical symplectic reflection
algebras (of which there are at most five, the maximum being obtained
exactly for $\bC[V^*]^G$ itself), and also give a bound (ten) on the
number of zero-dimensional symplectic leaves of these algebras.
\subsection{The parameters $c$ for which $eH_c(G) e$ is singular}
\begin{theorem}\label{t:main3}
  The locus of $c \in \msC$ such that $e H_c(G) e$ is singular is
  precisely the union of the following twenty-one hyperplanes:
\begin{itemize}
\item[(i)] The sixteen of the form $\sum_{s \in \cS} \chi(s) \cdot c(s)
  = 0$, where $\chi$ is a one-dimensional character (i.e., $\chi(s) =
  \pm 1$ for all $s$, $\chi(-\Id)=1$, 
  and $\prod_{i=1}^5 \chi(s_i) = 1$ for a choice
  of representatives $s_i$ of the conjugacy classes of $\cS$);
\item[(ii)] The five of the form $c(s) = 0$ for some $s \in \cS$
  (equivalently, $c(-s)=0$).
\end{itemize}
In the case of type (i), $H_c(G)$ admits the one-dimensional
representation $\chi$ with trivial action of $V^*$. In the case of
type (ii), $H_c(G)$ admits two two-dimensional families of
sixteen-dimensional irreducible representations whose $G$-structures
are isomorphic to $\Ind_{\{s, 1\}}^G \mathbf{1}$ and $\Ind_{\{ s , 1
  \}}^G \mathbf{sgn}$, respectively.
\end{theorem}
\begin{proof}
  Choose $c \in \msC$ such that $e H_c(G) e$ is not regular. Then
  $Z(H_c(G))$ is also not regular and we can choose a closed point
  $\psi : Z(H_c(G)) \rightarrow \bC$ lying in the singular locus of
  $\Spec Z(H_c(G))$. By Proposition \ref{p:cssra-smooth}, there exists
  an irreducible representation whose $G$-character $\chi$ is a proper
  subrepresentation of the regular representation.
  Then the parameter $c$ satisfies equations (\ref{e:hyp-eqns1}) and (\ref{e:hyp-eqns2}) and Lemma
  \ref{lem:complemma} implies that $c$ must lie in one of the
  twenty-one hyperplanes in the statement of the theorem.

  Conversely, if we choose $c$ to lie in one of these twenty-one
  hyperplanes we must show that there exists a representation of
  $H_c(G)$ of dimension less than $|G|$. One can easily check the
  claim of part (i). Therefore we concentrate on part (ii) and assume
  that $c(s) = 0$ for some symplectic reflection $s$. Being a
  symplectic reflection, $\dim V^s = 2$. Let $P$ be the parabolic
  subgroup of $G$ that is the stabilizer of a generic point of
  $V^s$. Then $P = \langle s \rangle \simeq \bZ_2$: if $p \in P$ then
  $V^s \subseteq V^p$ and either $p$ is a symplectic reflection or $p
  = \Id$. However, if $p$ is a symplectic reflection not equal to $s$
  ($ = s^{-1}$) then $ps$ is neither $\Id$ nor a symplectic
  reflection. Now, consider the symplectic reflection algebra
  $H_{c|_P}(P,(V^P)^\perp)$ defined by $P$, the restriction $c|_P$ of
  $c$ to $P$, and the symplectic vector space $(V^P)^\perp \subseteq
  V$.  Since $c(s) = 0$, $H_{c|_P}(P,(V^P)^\perp) = H_0(\bZ_2,
  \bC^2)$. There exist (up to isomorphism) exactly two one-dimensional
  representations of $H_0(\bZ_2, \bC^2)$, which we denote by
  $L(\mathbf{1})$ and $L(\mathbf{sgn})$, which are isomorphic to the
  trivial and sign representations, respectively, as $\bZ/2$-modules,
  and have the trivial action of $\bC^2$.
  Part (ii) now follows from Losev's Theorems \ref{t:Losev1} and
  \ref{t:Losev2}. In particular, the fact that there are
  two-dimensional families of representations of $H_c(G)$ isomorphic
  as $G$-modules to $\Ind_{\{s, 1\}}^G \mathbf{1}$ and $\Ind_{\{ s , 1
    \}}^G \mathbf{sgn}$, respectively, follows from the fact that
  there is a two-dimensional leaf of $Z_c(G)$ labeled by $(P)$ such
  that, at each point of the leaf, there are irreducible $H_c(G)$-modules
  supported at that point isomorphic as $G$-modules to $\Ind_{\{s,
    1\}}^G \mathbf{1}$ and $\Ind_{\{ s , 1 \}}^G \mathbf{sgn}$.
\end{proof}

Theorem \ref{t:main3} carries the following remarkable consequence, which is not
otherwise obvious:
\begin{corollary}
If $c(s) \equiv 1$ is the constant function, then $e H_c(G) e$ is smooth. 
\end{corollary}
\begin{proof}
  It is evident that $c(s) \equiv 1$ is not contained in any of the
  hyperplanes of type (ii) from Theorem \ref{t:main3}.  Also, since
  there are an odd number (5) of conjugacy classes of symplectic
  reflections $s \in \cS$, for every one-dimensional character $\chi$ of
  $G$, the number of occurrences of $+1$ among the values $\chi(s), s
  \in \cS$ is not equal to the number of occurrences of $-1$ (and these
  are the only values that occur, since $s^2 = \Id$ for all $s \in
  \cS$).  Hence, the constant function $c(s) \equiv 1$ is not contained in any
  hyperplanes of type (i). Thus, the result follows from Theorem
  \ref{t:main3}.
\end{proof}

The following lemma, which is required in the proof of Theorem \ref{t:main3}, is verified by computer\footnote{The Magma \cite{magma} code used to verify Lemma \ref{lem:complemma} can be obtained by emailing the authors.}. 

\begin{lemma}\label{lem:complemma}
  Let $\chi$ be the character of a $G$-submodule of the regular
  representation. Then the subspace of $\msC$ defined by equations
  (\ref{e:hyp-eqns1}) and (\ref{e:hyp-eqns2}) is contained in one of
  the twenty-one hyperplanes described in Theorem \ref{t:main3}.
\end{lemma}

\subsection{The singular locus of singular $e H_c(G) e$}
We can also deduce from the proof of Theorem \ref{t:main3} more information on
the singularities of the singular  $\Spec e H_c(G) e$:
\begin{corollary}\label{cor:twodimleaves}
  The number of two dimensional leaves in $\Spec e H_c(G)
  e$ equals the number of conjugacy classes of symplectic reflections
  $\{ s, - s\}$ such that $c(s) = 0$.
\end{corollary}
\begin{proof}
  As noted in the proof of Theorem \ref{t:main3}, the proper parabolic
  subgroups of $G$ are all of the form $\langle s \rangle$ for some
  symplectic reflection $s$. Therefore there is a natural bijection
  $\{ s, -s \} \mapsto (\langle s \rangle)$ between the conjugacy
  classes of symplectic reflections in $G$ and conjugacy classes of
  proper parabolic subgroups of $G$. Now Losev's Theorem
  \ref{t:Losev1} says that there is a bijection between height two
  Poisson prime ideals labeled by a conjugacy class $( \langle s
  \rangle )$ and the $\Xi$-orbits of maximal Poisson ideals in
  $Z_{c|_P}(\bZ_2, \bC^2)$. If $c(s) = 0$ then there is a unique maximal
  Poisson ideal in $Z_{0}(\bZ_2, \bC^2)$, which corresponds to the isolated
  singularity of $\bC^2 / \bZ_2$. If $c(s) \neq 0$, then there are
  no maximal Poisson ideals in $Z_{c|_P}(\bZ_2, \bC^2)$. Therefore $c(s) = 0$
  implies that there is a unique two-dimensional leaf in $\Spec e
  H_c(G) e$ labeled by $(\langle s \rangle )$ and $c(s) \neq 0$
  implies that there are no two-dimensional leaves labeled by
  $(\langle s \rangle )$.
\end{proof}
We can also give partial information on the zero-dimensional
symplectic leaves of $\Spec eH_c(G) e$.  Recall that, for a Poisson
algebra $A$, the zeroth Poisson homology is defined as $\HP_0(A) := A
/ \{A, A\}$, where $\{A, A\}$ is considered as a vector subspace of
$A$.  The space of Poisson traces is the dual vector space,
$\HP_0(A)^* = \{\phi: A \to \bC \mid \phi(\{a,b\}) = 0, \forall a,b
\in A\}$.  Recall also that, for each zero-dimensional symplectic leaf
$\{x\} \subseteq \Spec A$, evaluation at $x$ is a Poisson trace,
and these are linearly independent for distinct zero-dimensional
leaves.  Hence, the number of zero-dimensional symplectic leaves is at
most $\dim \HP_0(A)^*$.
\begin{proposition}\label{p:hp0-cssra}
For all commutative spherical symplectic reflection algebras
$A = eH_c(G) e$ deforming $\bC[V^*]^G$, the following holds:
\begin{enumerate}
\item[(a)] $\dim \HP_0(A) = 10$;
\item[(b)] $\Spec A$ has at most ten zero-dimensional symplectic leaves.
\end{enumerate}
\end{proposition}

Note that part (a) confirms \cite[Conjecture 1.3.5.(i)]{ESweyl} on
symplectic resolutions in this case, which states that, whenever $V^*/G$
admits a symplectic resolution for $G < \Sp(V)$, then $\dim
\HP_0(\bC[V^*]^G)$ equals the number of conjugacy classes of elements $g
\in G$ such that $g-\Id$ is invertible.

The result is perhaps surprising in that there is a very large number
of proper subrepresentations of the regular representation of $G$
($5\cdot 2^{16}-1=327679$), and these can all be extended to
representations of $H_c(G)$ at special values of $c$ depending on the
representation. Thus, in principle, at special values of $c$ many of
these could appear and be supported on many distinct zero-dimensional
symplectic leaves. However, we see above that there are nonetheless at
most ten zero-dimensional symplectic leaves at each value of
$c$. (Note that, for example, at $c=0$, all representations of $G$
occur, but there is only one zero-dimensional symplectic leaf.)

Note also that this result does not rely on the brute-force
computation underlying Lemma \ref{lem:complemma} (although it does
rely on a different computer computation, namely computing
$\HP_0(\bC[V^*]^G)$ up to a certain polynomial degree provided by
\cite{hp0bounds}).
\begin{proof}
  Using the methods of \cite[\S 4]{hp0bounds} and Magma code there, we
  computed that $\HP_0(\bC[V^*]^G)$ is ten-dimensional. Then, according to
  \cite[Remark 2.13]{hp0bounds}, ten is also an upper bound for $\dim
  \HP_0(A)$ for all $A = e H_c(G) e$. To show that $\dim
  \HP_0(A)$ is exactly ten for all such $A$, since the dimension is
  upper-semicontinuous, it suffices to show that it is ten-dimensional
  for generic $c$.  Since $\Spec A$ is generically a symplectic
  manifold, this is a consequence of \cite[Theorem 1.8.(ii)]{EGsra}
  together with the isomorphism $\HP_0(A) \cong H^{\dim \Spec A}(\Spec
  A)$ (as explained in \cite[Theorem 22.2.1]{Br}, more generally,
  $\HP_*(A) \cong H^{\dim \Spec A - *}(\Spec A)$ for symplectic $\Spec
  A$).
\end{proof}
To summarize, if $c$ does not lie on any of the twenty-one hyperplanes
of Theorem \ref{t:main3} then $\Spec eH_c(G)e$ is a smooth symplectic
manifold. If $c$ is a generic point of one of the sixteen hyperplanes
such that $c(s) \neq 0$ for all symplectic reflections $s$, then the
singular locus of $\Spec eH_c(G)e$ consists of a single point,
corresponding to a one-dimensional representation of $H_c(G)$ (with
trivial action of $V^*$). If $c$ lies on at least one of these sixteen
hyperplanes but does not lie on any of the five hyperplanes
$c(s)=c(-s)=0$ for $s$ a symplectic reflection, then the singular
locus is zero-dimensional and consists of at most ten points.  On the
other hand, if $c(s) = 0$ for some $s$ then in addition to the smooth
locus, there are also two-dimensional and zero-dimensional leaves,
with the number of two-dimensional leaves given by the number of
hyperplanes of the form $c(s)=0$ on which $c$ lies (this is,
obviously, at most five, with equality if and only if $e H_c(G) e =
\bC[V^*]^G$ itself), and again with at most ten zero-dimensional
leaves. A generic point on one of the five hyperplanes of the form
$c(s)=0$ has exactly this corresponding two-dimensional leaf, and no
other leaves aside from the open leaf.

We remark that we do not know how to compute precisely how many
zero-dimensional symplectic leaves there are, nor even if the maximum
of ten is attained for any $c$.  To do this seems like an interesting
problem (although it may be difficult, as it is analogous to
determining the number, if any, of finite-dimensional representations
admitted by a given quantization of $\bC[V^*]^G$).

\subsection{Condition (ii) of Proposition  \ref{p:cssra-smooth}} \label{ss:regrep} The following observation
was used in the proof of Proposition \ref{p:cssra-smooth}.(ii),
which we needed in the proof of Theorem \ref{t:main3} (in order to
make the computation reported in Lemma \ref{lem:complemma} tractable):
\begin{lemma}\label{lem:regrep}
  Let $\chi : Z( H_c(G)) \rightarrow \bC$ be a closed point of $\Spec
  Z(H_c(G))$. Then there exists a finite-dimensional $H_c(G)$-module
  $M$ such that $M$ is isomorphic to the regular representation as a
  $G$-module and $z \cdot m = \chi(z) m$ for all $m \in M, z \in
  Z(H_c(G))$.
\end{lemma}
\begin{proof}
  Let $\Rep_{\bC G}(H_c(G))$ denote the variety of
  homomorphisms $\phi : H_c(G) \rightarrow \mathrm{End}_{\bC}(\bC G)$,
  whose restriction to $G$ is the $G$-action of left
  multiplication. The group $\mathsf{Aut}_G(\bC G)$ acts on
  $\Rep_{\bC G}(H_c(G))$ by base change. It is shown in
  \cite[Theorem 3.7]{EGsra} that there exists an irreducible component
  $\Rep^{\circ}$ of $\Rep_{\bC G}(H_c(G))$ such that
  the map $\pi : \Rep_{\bC G}(H_c(G)) \rightarrow \Spec
  Z(H_c(G))$, sending a representation $M$ to the algebra homomorphism
  $Z(H_c(G)) \rightarrow \bC$ given by the action of $Z(H_c(G))$ on
  the line $e M \subseteq M$ (here $e = \frac{1}{|G|} \sum_{g \in G} g
  \in \bC[G]$ is the symmetrizer element), restricts to an isomorphism
  of varieties $\Rep^{\circ} // \mathsf{Aut}_G(\bC G)
  \stackrel{\sim}{\longrightarrow} \Spec Z(H_c(G))$. The irreducible
  component $\Rep^{\circ}$ is characterized as the closure in
  $\Rep_{\bC G}(H_c(G))$ of the set
  $\Rep^{\circ}_{\mathrm{reg}}$ of points in
  $\Rep_{\bC G}(H_c(G))$ that are irreducible
  $H_c(G)$-modules. Fix a representative $(M,\phi_M)$ in the unique
  closed $\mathsf{Aut}_G(\bC G)$-orbit in $(\pi
  |_{\Rep^{\circ}})^{-1}(\chi)$ and write $M = M_1 \oplus
  \cdots \oplus M_k$ for the decomposition of $M$ into irreducible
  $H_c(G)$-modules. Without loss of generality, suppose
  $e M \subseteq M_1$. Then $z \cdot m =
  \chi(z) m$ for all $m \in M_1$. We need to show that $z \cdot m =
  \chi(z) m$ for all $m \in M$. Fix $z \in Z(H_c(G))$ and consider the
  closed subvariety
$$
Y_z = \left\{ \phi(z) - \pi(\phi)(z) \Id_{\bC [G]} = 0 \ | \ \phi \in \Rep_{\bC G}(H_c(G)) \right\}
$$
of $\Rep_{\bC G}(H_c(G))$. Then $Y_z \cap
\Rep^{\circ}$ is closed in $\Rep^{\circ}$. On the
other hand, $\Rep^{\circ}_{\mathrm{reg}} \subseteq Y_z \cap
\Rep^{\circ}$ which implies that $\Rep^{\circ} \subseteq
Y_z$.  Since $(M,\phi_M) \in \Rep^\circ$, this implies that $\phi_M(z)
= \chi(z) \cdot \Id_M$, as desired.
\end{proof}


\appendix

\section{Summary of \cite{Los-csra}}

We summarize those results of \cite{Los-csra} that have been used in
this article (note that we only needed them for the proof of Theorem
\ref{t:main3}, not for the proof of Theorem \ref{t:main2} and hence
also Theorem \ref{t:main1}). Recall that a parabolic subgroup of $G$
is defined to be any subgroup that is the stabilizer of some vector $v
\in V$. Let $P$ be a parabolic subgroup of $G$. By definition, it is
normal in its normalizer $N_G(P)$. Let $\Xi := N_G(P)/P$ be the
quotient. The algebra $H_c(G)$ has a canonical filtration given by
placing $G$ in degree zero and $V$ in degree one. Then $Z_c(G)$
inherits this filtration by restriction and $\gr (Z_c(G)) \simeq
\bC[V^*]^G$. If $\mf{p} \subset Z_c(G)$ is the prime ideal defining the
closure of a symplectic leaf of $Z_c(G)$, then it is known, by
\cite[Theorem 2.8]{MartinoAssociated}, that $\gr (\mf{p})$ is a prime
ideal defining the closure of a symplectic leaf of $V^*/G$. Since the
leaves of $V^* / G$ are in bijection with conjugacy classes of
parabolic subgroups of $G$ (as recalled in \S \ref{ss:symleaves}),
the leaves in $Z_c(G)$ can also be labeled by conjugacy classes of
parabolic subgroups of $G$ (however, the same conjugacy class could
label several different leaves). Let $\mathsf{PSpec}_{(P)} Z_c(G)$
denote the set of all leaves in $Z_c(G)$ that are labeled by $(P)$,
considered as a subset of $\Spec Z_c(G)$. We fix a representative $P$
in each conjugacy class $(P)$. There is a unique
zero-dimensional leaf $\{ 0 \}$ in $V^*/G$; it is labeled by
$(G)$. 

Next, we consider the algebra $H_{c|_P}(P,(V^P)^\perp)$, the
symplectic reflection algebra defined by the subgroup $P$, the
restriction $c|_P$, and the subspace $(V^P)^\perp \subseteq V$. The
group $\Xi$ acts on $\Spec Z_{c|_P}(P,(V^P)^\perp)$. Let
$\mathsf{PSpec}_{(P)}^{\Xi} Z_{c|_P}(P,(V^P)^\perp)$ denote the set of
$\Xi$-orbits of zero-dimensional leaves in $\Spec
Z_{c|_P}(P,(V^P)^\perp)$. By \cite[Theorem 1.3.2 (4)]{Los-csra}:

\begin{theorem}\label{t:Losev1}
There exists a bijection $\mathsf{PSpec}_{(P)}^{\Xi} Z_{c|_P}(P,(V^P)^\perp)  \stackrel{1 : 1}{\longleftrightarrow} \mathsf{PSpec}_{(P)} Z_c(G)$.  
\end{theorem}

Now fix a leaf $\mc{L}$ in $Z_c(G)$. It is labeled by some conjugacy
class of parabolics, $(P)$ say. For closed points $p \in \Spec
Z_{c}(G)$ and $q \in \Spec Z_{c|_P}(P,(V^P)^\perp)$ corresponding to
maximal ideals $\mf{m}_p \subseteq Z_c(G)$ and $\mf{n}_q \subseteq
Z_{c|_P}(P,(V^P)^\perp)$, denote by $H_c(G)_p$ and
$H_{c|_P}(P,(V^P)^\perp)_q$ the finite dimensional quotients of
$H_c(G)$ and $H_{c|_P}(P,(V^P)^\perp)$ by the ideals generated by
$\mf{m}_p \subseteq Z_c(G)$ and $\mf{n}_q$, respectively. Then
\cite[Theorem 1.4.1]{Los-csra} says:

\begin{theorem}\label{t:Losev2}
  Let $p \in \mc{L}$. Then there exists a zero-dimensional leaf $\{ q
  \}$ in $\Spec Z_{c|_P}(P,(V^P)^\perp)$ and an
  isomorphism of finite-dimensional algebras
$$
\theta : H_c(G)_p \stackrel{\sim}{\longrightarrow}
\mathrm{Mat}_{|G/P|}(H_{c|_P}(P,(V^P)^\perp)_q)
$$
such that the corresponding equivalence of categories $\theta^* : H_{c|_P}(P,(V^P)^\perp)_q - \mathrm{mod} \stackrel{\sim}{\rightarrow} H_c(G)_p - \mathrm{mod}$ satisfies
$$
\theta^*(M) = \Ind_P^G M
$$
as $G$-modules. 
\end{theorem}

\begin{remark}
  The second part of Theorem \ref{t:Losev2} regarding $G$-module
  structures is not explicitly stated in \cite[Theorem
  1.4.1]{Los-csra}. However, it follows from the definition of the
  isomorphism of \cite[Theorem 2.5.3]{Los-csra} and \cite[Corollary
  5.4]{BellamyCupidal}.
\end{remark}

\bibliographystyle{amsalpha}
\bibliography{master}
\end{document}